\def\misajour{September 26, 2011}
 
\documentclass[10pt]{article} 
 
\usepackage{amsmath}
\usepackage{amssymb} 
\usepackage{fancyhdr}
\usepackage{amsthm}

\usepackage[applemac]{inputenc}
\usepackage[pdftex]{graphicx}
\usepackage{hyperref}

\newtheorem{theorem}{Theorem} [section]
 
\newtheorem{proposition}[theorem]{Proposition} 
\newtheorem{corollary}[theorem]{Corollary} 
\newtheorem{lemma}[theorem]{Lemma}

\newtheorem*{rem}{Remark}

\def\A{\mathbf{A}} 
\def\C{\mathbf{C}} 
\def\G{\mathbf{G}} 
\def\P{\mathbf{P}} 
\def\Q{\mathbf{Q}} 

\def\bu{\mathbf{u}}
\def\bv{\mathbf{v}}
\def\bV{\mathbf{V}}
\def\Z{\mathbf{Z}}
\newcommand{\ZK}{\Z_K}

\def\calE{\mathcal {E}}
\def\calO{\mathcal {O}}
\def\calS{\mathcal {S}}
\def\Pic{\mathrm {Pic}}
\def\Card{\mathrm {Card}}
\def\Id{{\mathrm{Id}}} 
\def\GL{{\mathrm {GL}}}
\newcommand{\OS}{O_S}
\newcommand{\OSprime}{O_{S'}}
\newcommand{\Etilde}{\widetilde{E}}
\def\pgoth{{\mathfrak p}}
\def\mgoth{{\mathfrak m}}
\def\mgothpgoth{\mgoth_{\pgoth}}
 
\def\house#1{\setbox1=\hbox{$\,#1\,$}%
\dimen1=\ht1 \advance\dimen1 by 2pt \dimen2=\dp1 \advance\dimen2 by 2pt
\setbox1=\hbox{\vrule height\dimen1 depth\dimen2\box1\vrule}%
\setbox1=\vbox{\hrule\box1}%
\advance\dimen1 by .4pt \ht1=\dimen1
\advance\dimen2 by .4pt \dp1=\dimen2 \box1\relax}

\def\denominator{d}

\begin{document}

\begin{center}
{\Large{\bf 
Some remarks on diophantine equations 
\\
and diophantine approximation }}

\medskip
 \bigskip

\large

Claude LEVESQUE and Michel WALDSCHMIDT
\end{center} 

\bigskip

\centerline{{\it Dedicated to professor H\`a Huy Kho\'ai.}} 
\bigskip

{\sc Abstract.} We first recall the connection, going back to A.~Thue, between rational approximation to algebraic numbers and integer solutions of some Diophantine equations. Next we recall the equivalence between several finiteness results on various Diophantine equations. We also give many equivalent statements of Mahler's generalization of the fundamental theorem of Thue. In particular, we show that the theorem of Thue--Mahler for degree $3$ implies the theorem of Thue--Mahler for arbitrary degree $\ge3$, and we relate it with a theorem of Siegel on the rational integral points of the projective line $\P^1(K)$ minus $3$ points. Finally we extend our study to higher dimensional spaces in connection with Schmidt's Subspace Theorem. 

{\sc Classification MSC 2010}: 11D59; 11J87; 11D25 


\section{Introduction}\label{S:Introduction}

The fundamental theorem of Thue obtained in 1908--1909 can be stated equivalently (Proposition $\ref{Proposition:Thue}$) as a result about the finiteness of the set of integral points on an algebraic curve, or as a result of diophantine approximation of algebraic numbers by rational numbers improving Liouville's inequality. Over a number field $K$, Thue's result on Diophantine equations is equivalent (Proposition $\ref{Proposition:detmKyoto}$) with finiteness statements on the number of integral points on Thue curves, Mordell curves, elliptic curves, hyperelliptic curves, superelliptic curves, and also to the finiteness of the set of solutions of the unit equation $E_1+E_2=1$, where the unknowns $E_1,E_2$ take their values in the group of units of $K$. 

In Proposition $\ref{Proposition:equivalence}$, we will give many equivalent statements of a generalization of this theorem of Thue by Mahler. In particular, we will show that the theorem of Thue--Mahler for degree $3$ implies the theorem of Thue for arbitrary degree $\ge3$, and we will relate it with a theorem of Siegel on the integral points of the projective line $\P^1(K)$ minus $3$ points. 
We remark that Siegel's theorem has been generalized by Vojta for the integral points on a projective variety minus a divisor. 
 Vojta’s proof rests on the Subspace theorem of Schmidt and comes also into play in the work of H\`a Huy Kho\'ai
 \cite{MR1247071,MR1474954}. 
 We shall use Vojta's result only in the special case where the variety is a projective space $\P^n(K)$ and the divisor is a union of hyperplanes, in which case it is equivalent to the finiteness of the set of solutions of a generalized $S$--unit equation
 (see Proposition $\ref{Proposition:EquivalenceUnitEqnHyperplans}$).

\section{Rational approximation and diophantine equations } \label{S:ApproximationEquations}

The following link, between the rational approximation on the one hand and the finiteness of the set of solutions of some diophantine equations on the other hand, happens to be well known thanks to the work of A. Thue.
 
\begin{proposition} \label{Proposition:Thue}
 Let $f\in\Z[X]$ be an irreducible polynomial of degree $d$ and let $F(X,Y)=Y^df(X/Y)$ be the 
 associated homogeneous binary form of degree $d$. Then the following two assertions are equivalent:
\par 
$(i)$ For any integer $k\neq 0$, the set of $(x,y)\in\Z^2$ verifying 
\begin{equation}\label{Equation: Thue(i)}
F(x,y)=k
\end{equation}
is finite. 
\par 
$(ii)$ For any real number $\kappa>0$ and for any root $\alpha\in\C$ of $f$, the set of rational numbers $p/q $ verifying 
\begin{equation}\label{Equation:Thue(ii)}
\left| \alpha-\frac{p}{q}\right| \le \frac{ \kappa }{q^{d}} 
\end{equation}
is finite.
\end{proposition} 

Condition $(i)$ can also be phrased by stating that for any positive integer $k$, the set of $(x,y)\in\Z^2$ verifying 
$$
0<|F(x,y)|\le k
$$
is finite.

Before proceeding with the proof, a few remarks are in order. When we consider an element $p/q\in\Q$, it should be understood that $p$ and $q$ are integers with $q>0$ and that  if $p=0$ then $q=1$. Moreover, the set defined in the assertion $(ii)$ would be the same if we added the condition $\gcd(p,q)=1$. 
\bigskip

In the case when $d=1$, the two assertions are false. As a matter of fact, if we write $f(X)=a_0X+a_1$ with $a_0\not= 0$, for $k=a_0$ the equation $a_0X+a_1Y=k$ has an infinite number of solutions $(x,y)$: 
$$
x=na_1+1,\quad y=-na_0 \qquad \hbox{with}\quad n\in\Z,
$$
and for $\kappa=|a_1|/a_0$ the root $\alpha=-a_1/a_0$ of $f$ has an infinite number of approximations $p/q$ satisfying $(\ref{Equation:Thue(ii)})$ with $\gcd(p,q)=1$, namely when 
$$
\frac{p}{q}=\frac{-na_1}{na_0-1}
$$
for all integers $n >0$ (with $n>1$ whenever $a_0=1$). 
\bigskip

In the case when $d=2$, the two assertions can be true, take for instance $f(X)=X^2+a$ with $a\in\Z$, $a>0$, 
and both of them can also be false, 
 take for instance $f(X)=X^2-a$ with $a\in\Z$, $a>0$ squarefree. For $d\ge 3$,
we know, since the work of Thue, that these two assertions are true. The statement in $(ii)$ with $d\geq 3$ is the first improvement of the Liouville inequality and  was obtained by Thue in a stronger form with $(\ref{Equation:Thue(ii)})$ replaced by 
$$
\left| \alpha-\frac{p}{q}\right| \le \frac{ \kappa(\epsilon) }{q^{(d/2)+1+\epsilon}} 
$$ 
for any $\epsilon>0$ 
(\cite{MR0249355}, Chap.~6 ; 
  \cite{MR0568710}, Chap.~V \S3; 	
 \cite{MR88h:11002}, Chap.~5; 
 \cite{SerreMordellWeil}, \S7.2; %
  \cite{UZ1}, Chap.~1, \S2; 
   \cite{1115.11034}; 
 \cite{UZ2}, Chap.~2). 
It gave birth to the works of C.L. Siegel, F. Dyson, Th. Schneider, K.F. Roth and W.M. Schmidt, culminating with the Subspace theorem, including a number of variations with a lot of applications
   (\cite{MR88h:11002}, Chap.~5; 
  \cite{MR93a:11048}, Chap.~IX \S7; 
 \cite{SerreMordellWeil}, \S7.2;  %
 \cite{UZ1}, Chap.~1, \S6;  
 \cite{1115.11034}). 

 The proof of Proposition $\ref{Proposition:Thue}$ is effective: from an explicit upper bound for the heights of the exceptions $(x,y)$ in statement $(i)$, one deduces
an explicit upper bound for the exceptions $q$ in statement $(ii)$, and conversely. Such explicit upper bounds are known

 \smallskip
 
\par
\begin{proof}[Proof of Proposition $\ref{Proposition:Thue}$]
Write 
$$
f(X)=a_0X^d+a_1X^{d-1}+\cdots + a_{d-1}X+a_d
$$
and
$$
F(X,Y)= a_0X^d+a_1X^{d-1}Y+\cdots + a_{d-1}XY^{d-1}+a_dY^d.
$$
Without loss of generality we may assume $a_0>0$. 
\par 
$(1)$ Suppose now that the assertion $(i)$ is true. Consider a root $\alpha$ of $f$, a number $\kappa >0$ 
and a rational number $p/q$ verifying $(\ref{Equation:Thue(ii)})$. Without loss of generality we can suppose $q^d\ge \kappa$. 
We have 
$$
F(X,Y)=a_0\prod_{\sigma} (X-\sigma(\alpha)Y),
$$
where $\sigma$ in the product runs through the set of embeddings of the field $K:=\Q(\alpha)$ in $\C$. 
The element $\alpha$ is in $\C$ and we write $\Id $ for the inclusion of $K$ into $\C$. Hence 
$$
|F(p,q)|= a_0q^d \left|\alpha-\frac{p}{q}\right|\prod_{\sigma\neq \Id} \left|\sigma(\alpha)-\frac{p}{q}\right|.
$$
For $\sigma\neq \Id$, we use the upper bound 
$$
 \left|\sigma(\alpha)-\frac{p}{q}\right|\le |\alpha-\sigma(\alpha)|+ \left| \alpha-\frac{p}{q}\right| 
 \le |\alpha-\sigma(\alpha)| +1,
 $$
which comes from $(\ref{Equation:Thue(ii)})$ and from $q^d\ge k$. Therefore 
$$
0<|F(p,q)|\le a_0 \kappa \prod_{\sigma\neq \Id} \bigl( |\alpha-\sigma(\alpha)| +1 \bigr).
$$
The assertion $(i)$ allows us to conclude that the set of elements $p/q$ is finite, from which we deduce the assertion $(ii)$.
 \bigskip 
\par 
$(2)$ Conversely, suppose that the assertion $(ii)$ is true. Let $k$ be a non--zero integer and let $(x,y)\in\Z^2$ satisfy $F(x,y)=k$. We want to show, by assuming $(ii)$, that these couples $(x,y)$ belong to a finite set. Without loss of generality, we may suppose $|y|$ sufficiently large. Let $\alpha$ be a root of $f$ at a minimal distance from $x/y$. We remark that
$$ 
|k|\;=\;|F(x,y)|\; = \;a_0 |y|^d
 \left|\alpha-\frac{x}{y}\right|\prod_{\sigma\neq \Id} \left|\sigma(\alpha)-\frac{x}{y}\right|
\;\geq \;
 a_0 |y|^d \left|\alpha-\frac{x}{y}\right|^d,
$$
whereupon 
$$
 \left|\alpha-\frac{x}{y}\right|^d \leq \frac{|k|}{a_0|y|^d} \cdotp
$$
Therefore, for $|y|$ sufficiently large, for instance with
$$ 
|y|^d \geq \frac{ 2^d |k|}{a_0\, \displaystyle \min_{\sigma\neq \Id} ( |\alpha-\sigma(\alpha)|^d)},
$$
we come up with the inequality
$$ 
\left|\alpha-\frac{x}{y}\right|\; \leq \; \frac12 \min_{\sigma\neq \Id} ( |\alpha-\sigma(\alpha)|,
$$
which allows us to deduce that for any $\sigma\neq \Id$, we have 
 $$
 \left|\sigma(\alpha)-\frac{x}{y}\right|\;\geq \; \frac{1}{2} |\alpha-\sigma(\alpha)|.
 $$
Since $f$ is irreducible, 
$$
f'(\alpha)=a_0 \prod_{\sigma\neq \Id} \bigl( \alpha-\sigma(\alpha) \bigr)\neq 0.
$$
 Hence we deduce 
$$
|k|\;=\; |F(x,y)| \;=\; a_0 |y|^d
 \left|\alpha-\frac{x}{y}\right|\prod_{\sigma\neq \Id} \left|\sigma(\alpha)-\frac{x}{y}\right|
\; \geq \;2^{-d+1} |y|^d |f'(\alpha)|\cdot 
 \left|\alpha-\frac{x}{y}\right| ,
 $$
from which we come up with
$$
\left|\alpha-\frac{x}{y}\right| \; \leq \;\frac{\kappa}{ |y|^d } 
\quad \hbox{with}\quad
\kappa=\frac{2^{d-1} |k| }{|f'(\alpha)|}\cdotp
$$
 From the the assertion $(ii)$, we can say that the set of rational numbers $x/y$ verifying this inequality is finite. This allows us to conclude that the assertion $(i)$ is true. 
 \end{proof}

\section{Diophantine equations and unit equations}\label{S:texteJapon}
In section $\ref{S:ApproximationEquations}$, we considered the basic situation of rational numbers and points with rational integer coordinates on Thue curves. Here we consider the algebraic numbers while the number field $K$ may vary. We denote by $\ZK$ the ring of algebraic integers of $K$ and by $\ZK^\times$ the unit group of $K$.
Let us quote some results whose proofs appear in \cite{W:NotesHirata}.

\begin{proposition}\label{Proposition:detmKyoto}
The following statements are equivalent:
\hfill\break
\null\quad
$\bullet$
\hbox{\rm (M)} For any number field $K$ and for any non--zero element $k$ in $K$, the Mordell equation 
$$
Y^2=X^3+k
$$ 
has but a finite number of solutions $(x,y)\in \ZK\times\ZK$.
\hfill\break
\null\quad
$\bullet$
\hbox{\rm(E)} For any number field $K$ and for any polynomial $f$ in $K[X]$ of degree $3$ with three distinct complex roots, the elliptic equation 
$$
Y^2=f(X)
$$ 
has but a finite number of solutions $(x,y)\in \ZK\times\ZK$.
\hfill\break
\null\quad
$\bullet$
\hbox{\rm(HE)} For any number field $K$ and for any polynomial $f$ in $K[X]$ with at least three simple complex roots, the hyperelliptic equation 
$$
Y^2=f(X)
$$ 
has but a finite number of solutions $(x,y)\in \ZK\times\ZK$.
\hfill\break
\null\quad
$\bullet$
\hbox{\rm(SE)} For any number field $K$, for any integer $m\ge 3$ and for any polynomial $f$ in $K[X]$ with at least two distinct complex roots whose orders of multiplicity are prime to $m$, the superelliptic equation 
$$
Y^m=f(X)
$$ 
has but a finite number of solutions $(x,y)\in \ZK\times\ZK$.
\hfill\break
\null\quad
$\bullet$
\hbox{\rm(T)} For any number field $K$, for any non--zero element $k$ in $K$ and for any elements $\alpha_1,\ldots,\alpha_n$ in $K$ with $\Card\{\alpha_1,\ldots,\alpha_n\}\ge 3$, the Thue equation 
$$
(X-\alpha_1Y)\cdots (X-\alpha_nY) =k
$$
has but a finite number of solutions $(x,y)\in \ZK\times\ZK$.
\hfill\break
\null\quad
$\bullet$
\hbox{\rm(S)} For any number field $K$ and for any elements $a_1$ and $a_2$ in $K$ with $a_1a_2\neq 0$, the Siegel equation 
$$
a_1E_1+a_2E_2=1
$$ 
has but a finite number of solutions $(\varepsilon_1,\varepsilon_2)\in \ZK^\times\times\ZK^\times$.

\end{proposition}

Each of these statements is a theorem: the first four ones are due to Siegel who proved that the sets of integral points
respectively on a Mordell curve (M), on an elliptic curve (E), on a hyperelliptic curve (HE), on a superelliptic curve (SE), are finite. Statement (T) is due to Thue and (S) deals with the unit equation introduced by Siegel. 

For each of the six equivalent statements in Proposition $\ref{Proposition:detmKyoto}$, an upper bound is known for the size of the solutions; the proofs of the equivalences between them are elementary and effective: they allow one to deduce, from an explicit version of any of these statements, an explicit version of the other ones.

The proof of the equivalence given in \cite{W:NotesHirata} is elementary; it goes as follows:
$$
\begin{matrix}
\hbox{\rm(SE)} & \Longrightarrow & \hbox{\rm(M)} & \Longleftarrow & \hbox{\rm(E)}
\\
 \Uparrow&& \Downarrow&& \Uparrow
 \\
\hbox{\rm(T)} & \Longleftarrow & \hbox{\rm(S)} & \Longrightarrow & \hbox{\rm(HE)}
\end{matrix}
$$
The three implications which are not so easy to prove are 
$$
\hbox{\rm(T)} \Longrightarrow \hbox{\rm(SE)},
\quad
\hbox{\rm(S)} \Longrightarrow \hbox{\rm(T)}
\quad
\hbox{and}\quad 
\hbox{\rm(S)} \Longrightarrow \hbox{\rm(HE)}.
$$
Further statements are equivalent to each of the statements of Proposition $\ref{Proposition:detmKyoto}$; one of them is Siegel's Theorem on the finiteness of integral points on a curve of genus $1$ (of which $(E)$ is only a special case) -- see
  \cite{MR0249355}, Chap.~28, Th.~2; 
\cite{BakerTNT},  Chap.~4; 
 \cite{LangECDA}, Chap.~VI (see in particular the appendix);  
\cite{57:12383}, Chap.~3; 
 \cite{MR88h:11002}, Chap.~5 and 6; 
 \cite{SerreMordellWeil}, Chap.~7 and 8; %
 \cite{UZ2}, Chap.~2. 

\section{Projective spaces, places, $S$--integers} 
 \label{S:Background}
 
We recall here some basic facts on projective spaces, on places of a number field, on $S$--integers and $S$--units, and finally on the notion of $S$--integral points. 
 
\subsection{Projective spaces}

Let $E$ be a $K$--vector space of finite dimension. The {\it projective space $\P(E)$ of $E$} is the set of equivalence classes of elements in $E\setminus\{{\bf 0}\}$ 
for the following equivalence relation: for $ {\bf v}$ and ${\bf v}'$ in $E$, 
$$
\hbox{${\bf v}\equiv {\bf v}'$ if and only if 
there exists $t\in K^\times$ with ${\bf v}'=t{\bf v}$. }
$$
In other terms, $\P(E)$ is the set of lines (one--dimensional vector subspaces) of $E$. 
A {\it linear projective subspace} of $\P(E)$ is a subset of the form $\P(E')$ where $E'$ is a vector subspace of $E$. If $E'$ is a $2$--dimensional subspace (resp. a hyperplane) of $E$, then $\P(E')$ is called {\it a projective line} (resp. {\it a projective hyperplane}) of $\P(E)$.

If the $K$--vector space $E$ has dimension $n+1$, the {\it dimension} of the projective space $\P(E)$ is $n$ by definition. A {\it projective line} is a projective space of dimension $1$, a {\it projective plane} is a projective space of dimension $2$. Further, if $\{e_0,\ldots,e_n\}$ is a basis of $E$, the class $P$ of $x_0e_0+\cdots+x_ne_n$ in $\P(E)$ is denoted by $(x_0:x_1:\cdots:x_n)$, and we say that the {\it projective coordinates of} $P$ are $(x_0:x_1:\cdots:x_n)$. The choice of a basis of $E$ determines a system of  projective coordinates $(X_0:\cdots:X_n)$ on $\P(E)$. 

When $E$ is the vector space $K^{n+1}$, we write $\P^n(K)$ instead of $\P(K^{n+1})$. 
Therefore, using the canonical basis of $K^{n+1}$, we identify $\P^n(K)$ with the set of classes
of $(n+1)$-tuples $ (x_0, x_1, \dots , x_n)$ of $K^{n+1} \setminus \{ {\bf 0}\}$ modulo the equivalence relation:
$ (x_0, x_1, \dots , x_n) \equiv (x'_0, x'_1, \dots , x'_n)$ if and only if there exists $t\in K^\times $ such that $x'_i=tx_i$ for $i=0,\dots, n$. The class of $ (x_0, x_1, \dots , x_n) $ in $\P^n(K)$ will then be denoted by $(x_0: x_1:\cdots : x_n)$. The choice of a basis of $E$ determines a system of projective coordinates $(X_0:\cdots:X_n)$ on $\P^n$; a change of basis of $E$, given by a matrix in $\GL_{n+1}(K)$, produces another system of projective coordinates $(Y_0:\cdots:Y_n)$ on $\P^n$.

 \subsection{Places, $S$--integers, $S$--units}
We first recall some basic facts on places of number fields. There is a bijection between the set of ultrametric places of $K$ and the set of prime ideals of the ring $\calO=\ZK$ of integers of $K$, where the place $v$ corresponds to the prime ideal $\pgoth$ of $\calO$ so that 
$$
\pgoth
=\; \left\{ x\in \calO \, \mid \, |x|_v < 1 \right\}.
$$
The localization of ${\calO} $ at ${\pgoth}$, 
$$
{\calO} _{\pgoth} \;=\; \left\{\frac{a}{b} \, \mid \, a\in {\calO} , b\in \calO \setminus {\pgoth}\right\} \; =\; \left\{ x\in K \, \mid \, |x|_v \leq 1 \right\},
$$
is a local ring, with maximal ideal 
$$
\mgothpgoth\;=\; {\pgoth}{\calO} _{\pgoth}\;=\; \left\{\frac{a}{b} \, \mid \, a\in {\pgoth} , b\in \calO \setminus {\pgoth}\right\}
\;=\; \{ x\in K\ \, \mid \, |x|_v < 1 \}.
$$
The {\it residue field} of ${\calO} _{\pgoth}$ is $\kappa_{\pgoth}:={\calO} _{\pgoth}/ \mgothpgoth$.
We denote by $\pi_{\pgoth}$ the canonical surjective homomorphism 
${\calO} _{\pgoth} \rightarrow \kappa_{\pgoth}$ with kernel $\mgothpgoth$.
The unit group of ${\calO} _{\pgoth}$ is 
$$
{\calO} _{\pgoth}^\times = {\calO} _{\pgoth}\setminus \mgothpgoth\; = \; \pi_{\pgoth}^{-1}(\kappa_{\pgoth}^\times)
\; =\; \left\{\frac{a}{b} \, \mid \, a , b\in \calO \setminus {\pgoth}\right\}
\;=\; \{ x\in K \, \mid \, |x|_v = 1 \}.
$$
We shall use also the notations ${\calO} _v$, $\mgoth_v$, $\kappa_v$, $\pi_v$ when $v$ is the place associated with $\pgoth$.

Let $P$ be a point in $\P^n(K)$ and $v$ an ultrametric place of $K$. We select projective coordinates $(x_0:\cdots:x_n)$ of $P$. Let $i_0\in\{0,\ldots,n\}$ satisfy $|x_{i_0}|_v=\max_{0\le i\le n} |x_i|_v$. For $i=0,\dots,n$, set $y_i=x_i/x_{i_0}$. Then 
$(y_0:\cdots:y_n)$ is a system of projective coordinates of $P$ with $y_i\in \calO_v$ and $y_0, \dots,y_n$ not all in $\mgoth_v$. Hence $(\pi_v(y_0):\cdots:\pi_v(y_n))$ is a system of projective coordinates of a point in $\P^n(\kappa_v)$ which will be called {\it the reduction, in the projective space on the residue field, of the point $P$}.

 
 We now introduce the definitions of the ring of $S$--integers and the group of $S$--units of a number field $K$, when $S$ is a finite set of places of $K$ including the archimedean places (see for instance 
\cite{MR88h:11002}, 
 Chap.~7;
\cite{SerreMordellWeil}, \S 7.1; %
\cite{UZ1}, 
 \S 3.3.2).
The ring $\OS$ of $S$-integers of $K$ is defined by
$$
\OS\;=\; \{ x\in K \, \mid \, |x|_v \leq 1 \, \mbox{ for each } \, v \not\in S \}= \bigcap_{{v}\not\in S} {\calO} _{v} .
$$ 
The group $\OS^\times$ of $S$-units of $K$ is the group of units of $\OS$, namely
$$
\OS^\times \;=\; \{ x\in K \, \mid \, |x|_v = 1 \, \mbox{ for each } \, v \not\in S \}= \bigcap_{{v}\not\in S} {\calO} _{v}^\times .
$$ 
Thanks to the last formulas, when we will deal with $S$-integers $\alpha$ (resp. $S$-units $\varepsilon$), 
we will use the fact that $\alpha$ (resp. $\varepsilon$) belongs 
to the local rings $ \calO_v$ (resp. to the unit groups of the local rings $ \calO_v$) at all places $v$ outside $S$.

Consider the special case $K=\Q$. The set $S$ is then the union of the infinite place of $\Q$ and finitely many ultrametric places. These ultrametric places are associated with prime numbers $p_1,\ldots,p_s$. The ring of $S$--integers consists of rational numbers of the form $a/b$ where the denominator $b$ has all its prime factors in the set $\{p_1,\ldots,p_s\}$, while the group of $S$--units consists of all rational numbers of the form $\pm p_1^{a_1}\cdots p_s^{a_s}$ with $a_1,\dots,a_s$ in $\Z$.


\subsection{$S$--integral points}

 There is a general notion of {\it set of integral points on a projective variety relative to a very ample effective divisor } 
(see for instance 
\cite{MR883451}, 
Chap.~1, \S4). We will deal with the very special case of this situation where the variety is a projective space $\P^n(K)$ and the divisor is a union of finitely many hyperplanes. For this special case, see also \cite{UZ1}, 
  Remark 3.14.

Let $S$ be a finite set of places of $K$ including the archimedean places. Let us take $(X:Y)$ for a system of projective coordinates  on $\P^1(K)$. A point of $\P^1(K)$ which is not $(1:0)$ has projective coordinates $(\alpha:1)$ for some $\alpha\in K$. By definition, this point is called 
{\it an $S$--integral point of $\P^1(K)\setminus\{(1:0)\}$} if and only if $\alpha$ is an $S$--integer. It is clear that if $\alpha$ is an $S$--integer, then, for each place $v$ not in $S$, it reduces, in the projective line on the residue field, to a point which is not $(1:0)$. The converse is true. Indeed, if $\alpha$ is not an $S$--integer,
 then there is a place $v$ of $K$ not in $S$ such that $|\alpha|_v>1$. For this place $v$ the reduction of $(\alpha:1)=(1:\alpha^{-1})$, in the projective line on the residue field, is $(1:0)$.
 
 Suppose now  that the projective coordinates of an $S$–integral point of $\P^1(K)\setminus \{(1 : 0)\}$ are $(u:1)$.
Then this point is also an $S$--integral point of $\P^1(K)\setminus \{(0: 1)\}$  if and only if, for each place $v$ not in $S$, it reduces, in the projective line on the residue field, to a point which is not in $(0 : 1)$, hence if and only if $u$ is an $S$–unit. If these conditions are satisfied, then the same point $(u:1)$ is also an $S$–integral point on
$\P^1(K)\setminus \{(1 : 1)\}$   if  and only if $u-1$ is an $S$–unit of $K$. 

In the same way, a point of $\P^n(K)$ which is not in the hyperplane $H_0$ of equation $X_0=0$ has coordinates $(1:\alpha_1:\cdots:\alpha_n)$. By definition, it is an {\it $S$--integral point of $\P^n(K)\setminus H_0$} if and 
only if $\alpha_1,\ldots,\alpha_n$ are in $\OS$. This is equivalent to the fact that, for each place $v$ not in $S$, it reduces, in the projective space $\P^n(K)$ on the residue field, to a point which is not in  $H_0$. Further, for $1\le i\le n$, denote by $H_i$ the hyperplane of equation $X_i=0$. Then the point $(1:\alpha_1:\cdots:\alpha_n)$ is an {\it $S$--integral point of $\P^n(K)\setminus (H_0\cup\cdots \cup H_n)$} if and 
only if $\alpha_1,\ldots,\alpha_n$ are in $\OS^\times$. Furthermore, if these conditions are satisfied, then the same point is an $S$--integral point on the complement of the hyperplane of equation $X_0+\cdots+X_n=0$ if and only if $1+\alpha_1+\cdots+\alpha_n$ is an $S$--unit.

\noindent
{\sc Examples.}
 Here are a few examples,  where we take some  systems of  projective coordinates  $(X_0:\cdots:X_n)$ on $\P^n(K)$, $(X:Y)$ on $\P^1(K)$ and $(T:X:Y)$ on $\P^2(K)$. 
\par
$\bullet$ The complement of a hyperplane in the projective space $\P^n(K) $ is an affine space, isomorphic to $\A^n(K)$. For instance 
$$
 \P^n(K) \setminus \{X_0=0\}= \{(1:x_1:\cdots : x_n) \,\mid \, (x_1,\dots ,x_n) \in K^n\}\; \simeq\; K^n ,
 $$
 and the set of $S$--integral points on $ \P^n(K) \setminus \{X_0=0\}$ can be identified with $\OS^n$. 
 \par
$\bullet$ The special case $n=1$ of the previous example consists in removing one point on the projective line $\P^1(K)$: one gets the affine line $\A^1(K)$, which is also the additive group $\G_a$, so
$$
 \P^1(K) \setminus \{(0:1)\} \simeq \G_a(K)=K; 
 $$
 if we remove two points from $\P^1(K)$,
 we obtain the multiplicative group $\G_m$, so
$$ 
\P^1(K) \setminus (\{(0:1)\; , (1:0) \}) = \{(x:1) \,\mid \, x \in K^\times \} \simeq \G_m(K)=K^\times,
$$
which is isomorphic to the affine variety $V:= \{(x,y)\in K^2 \,\mid \, xy=1 \}$, an isomorphism being given by  $(x:1)\longmapsto (x,x^{-1})$. In view of this isomorphism, given the fact that the set of $S$--integral points on $V$ is $V\cap \OS^2$, it follows that the set of $S$--integral points on $\G_m$
is $\OS^\times$.
\par
$\bullet$ 
If one removes from $ \P^1(K) $ three points, say $(0:1)$, $(1:0)$, $(1:-1)$, the set of $S$--integral points is the set of pairs $(\varepsilon_1,\varepsilon_2)$ of $S$--units such that $\varepsilon_1+\varepsilon_2$ is a unit. 
\par
$\bullet$ 
The complement of two distinct hyperplanes (lines) in the projective plane $\P^2(K) $ is isomorphic to the product of the multiplicative group by the additive group, 
$$ 
\P^2(K) \setminus (\{T=0\}\cup \{X=0\}) = \{(1:x:y) \,\mid \, (x,y) \in K^\times \times K \}\; \simeq\; K^\times \times K,
$$
and the set of $S$--integral points 
 can be identified with $\OS^\times\times \OS$. 
\par
$\bullet$ 
The complement in $\P^2(K)$ of three hyperplanes in general position,
$$
\P^2(K) \setminus (\{T=0\}\cup \{X=0\}\cup \{Y=0\}) = \{(1:x:y) \,\mid \, (x,y) \in K^\times \times K^\times \}\; \simeq\; K^\times\times K^\times,
$$
is isomorphic to the product of two copies of the multiplicative group, the integral points of which are $\OS^\times\times \OS^\times$. 
\par
$\bullet$ 
Consider the complement in $\P^2(K)$ of four hyperplanes in general position: 
$$
 {\cal T}:= \P^2(K) 
\setminus (\{T=0\}\cup 
\{X=0\}\cup \{Y=0\}\cup \{X+Y=0\}).
$$
Then ${\cal T}$ is an affine variety,
$$ {\cal T}
=\{(1:x:y)\, \mid \, (x,y) \in K^\times \times K^\times, x+y\neq 0 \}\; \simeq\; \{ (x,y) \in K^\times\times K^\times \,\mid \, x+y\neq 0\} ,
$$ 
isomorphic to
$$
V\;=\; \{(x,a,y,b,c)\in K^5 \, \mid \, ax=by=c(x+y)=1\},
$$
the bijection being given by $(x,y) \mapsto (x, x^{-1}, y, y^{-1}, (x+y)^{-1})$. Therefore the set of $S$-integral points on $V$ is $V\cap \OS^5$, whereupon the set of $S$--integral points on ${\cal T}$ is
$$ 
\{(x:y:1)\, \mid \, x, y, x+y \in \OS^\times \}.
$$
This completes our list of examples.

Dealing with the standard hyperplanes associated with a given system of projective coordinates, as we have done so far, allowed us to give an elementary introduction to the subject. We shall need to deal with the more general case of hyperplanes in $\P^n(K)$. We proceed in two stages. 

For the first one, we assume that the ring $\OS$ is principal, which enables us to work globally. Consider a hyperplane $H$ in $\P^n$. It has an equation
$$
a_0X_0+a_1X_1+\cdots+a_nX_n=0 \quad \mbox{with} \; a_i \in \OS\; (i=0,1, \dots , n),
$$
 which is unique up to multiplication by an element of $\OS^\times$, such that $\gcd(a_0,a_1,\dots,a_n)=1$. Further, any projective point $P$ in $\P^n(K)$ has projective coordinates 
$$
(x_0:x_1:\cdots:x_n) \quad  \mbox{with} \; x_i \in \OS\; (i=0,1, \dots , n) \;\mbox{ and } \;   \gcd(x_0,x_1,\dots,x_n)=1,
$$
and again such projective coordinates are unique up to multiplication by an element in $\OS^\times$. Then, by definition, $P$ is an $S$--integral point on $\P^n(K)\setminus H$ if and only if $a_0x_0+a_1x_1+\cdots+a_nx_n$ is an $S$--unit. 
 
 In the second and final stage of our definition, we remove the assumption that $\OS$ is principal. In this general case we work locally. Let again $H$ be a hyperplane of $\P^n(K)$ and $P$ a point of $\P^n(K)$ not in $H$. Let $v$ be an ultrametric place of $K$ not in $S$. Then $H$ has an equation
$$
a_0X_0+a_1X_1+\cdots+a_nX_n=0,
$$
with $a_i\in \OS$, $\max\{ |a_0|_v,|a_1|_v,\dots,|a_n|_v\}=1$ and $P$ has projective coordinates  
$$
(x_0:x_1:\cdots:x_n) \quad  \mbox{with} \; x_i \in \OS\; (i=0,1, \dots , n) \;\mbox{ and } \;   \max\{| x_0|_v,|x_1|_v\dots,|x_n|_v\}=1.
$$
This equation and these coordinates may depend on $v$. 
Then, by definition, $P$ is an $S$--integral point on $\P^n(K)\setminus H$ if and only if $|a_0x_0+a_1x_1+\cdots+a_nx_n|_v=1$ for all $v$ not in $S$. 

If one allows a finite extension of $S$ (as we will always do), one may work globally and use a single equation independent of $v$ as follows. Given a hyperplane $H$ of equation $a_0X_0+a_1X_1+\cdots+a_nX_n=0$ with $(a_0,\dots,a_n)\in K^{n+1}\setminus\{{\bf 0}\}$, one replaces $S$ by the union $S'$ of $S$ with the the finitely many places $v$ of $K$ such that $\max\{ |a_0|_v,|a_1|_v,\dots,|a_n|_v\}\not =1$. Then one uses this equation for $H$ for all $v\not\in S'$. 

Our definition depends on a choice of a system of  projective coordinates. If $(X_0:X_1:\cdots:X_n)$ and $(Y_0:Y_1:\cdots:Y_n)$ are two distinct systems of projective coordinates, then $S$--integral points in the first system may not be $S$--integral points in the second system. However, there is a matrix in $\GL_{n+1}(K)$ which links the two systems of projective coordinates, and if one defines $S'$ as the union of $S$ with the finitely many ultrametric places $v$ of $K$ such that the determinant $\Delta$ of this matrix satisfies $|\Delta|_v\not=1$, then a set of $S'$-integral points relative to one system of coordinates remains a set of $S'$--integral points  relative to the other. 

Since all our results will allow a finite extension of $S$, we shall work with this notion of $S$--integral points depending on a choice of coordinates. There is an alternative definition, which gives equivalent results in our situation, and has the advantage of yielding  the more general notion of $S$--integral points on affine varieties, where one allows bounded denominators (see {\it e.g.} 
\cite{MR2438848}, 
p.~259--260); this is what Serre calls {\it quasi--integral sets on an affine variety} in \cite{SerreMordellWeil}, \S7.1 and \S8.

\section{Thue, Mahler, Siegel, Vojta} 
 \label{S:Equivalence}
 
The aim of this section is to establish an equivalence between many assertions. The first two concern Thue--Mahler equations; we prove the very interesting fact that it suffices to solve the equation for the very special case of the cubic form $XY(X-Y)$ in order to deduce the general case. The next assertion is a theorem of Siegel on the finiteness of the number of solutions of an equation of the form $E_1+E_2=1$ in $S$--units $\varepsilon_1,\varepsilon_2$ of a number field. The fourth (resp{.} fifth) assertion is the particular case $n=1$ (resp{.} $n=2$) of the theorem 
stating that any set of $S$--integral points of $\P^n(K)$ minus $n+2$ hyperplanes is contained in an algebraic hypersurface, which is a special case of a more general result due to Vojta.

\bigskip

We will consider an algebraic number field $K$ and a finite set $S$ of places of $K$ containing all the archimedean places. Moreover $F$ will denote a binary homogeneous form with coefficients in $K$. 
We will consider the Thue--Mahler equations 
$F(X,Y)=E$ where the two unknowns $X, Y$ take respectively values $x, y$ in a given set of $S$–integers of $K$
while the unknown $E$ takes its values $\varepsilon$ in the set of $S$–units of $K$.
If $(x,y,\varepsilon)$ is a solution and if $m$ denotes the degree of 
$F$, then, for all $\eta\in\OS^\times$, the triple $(\eta x,\eta y, \eta^m \varepsilon)$ is also a solution. 

 \medskip

\noindent
{\bf Definition.} 
Two solutions $(x,y,\varepsilon)$ and $(x',y',\varepsilon')$ in $\OS^2\times\OS^\times $ of the equation $F(X,Y)=E $ are said to be {\it equivalent modulo} $\OS^\times$ if the points of $\P^1(K)$ with projective coordinates $(x:y)$ and $(x':y')$ are the same.

If the two solutions $(x,y,\varepsilon)$ and $(x',y',\varepsilon')$ are equivalent, there exists $\eta\in K^\times$ such that $x'=\eta x$ and $y'=\eta y$. Since $(x,y,\varepsilon)$ and $(x',y',\varepsilon')$ are solutions of the equation $F(X,Y)=E$, we also have $\varepsilon'=\eta^m \varepsilon$ where $m$ is the degree of the binary homogeneous form
$F(X,Y)$. Since $\varepsilon$ and $\varepsilon'$ are $S$--units, $\eta^m $ is also an $S$--unit, hence 
 $\eta\in\OS^\times$. In other terms, two solutions $(x,y,\varepsilon)$ and $(x',y',\varepsilon')$ are equivalent if there exists $\eta\in \OS^\times$ such that 
$$
x'=\eta x,\quad y'=\eta y, \quad \varepsilon'=\eta^m \varepsilon.
$$ 

 \medskip

\noindent
{\bf Definition.} 
We will say that such a Thue--Mahler equation has {\it but a finite number of classes of solutions} if the set of solutions $(x,y,\varepsilon)\in\OS^2\times\OS^\times $ can be written as the union of a finite number of equivalence classes modulo $\OS^\times$. 

This last definition is equivalent to saying that the set of points $(x:y)$ of $\P^1(K)$, for which there exists 
 $\varepsilon \in \OS^\times $ such that $(x,y,\varepsilon)$ is a solution, is finite. 
 
\begin{proposition}\label{Proposition:equivalence} 
Let $K$ be an algebraic number field.
 \hfill\break 
$(1)$
The following four assertions are equivalent: 
 \hfill\break\null\quad
$(i)$ For any finite set $S$ of places of $K$ containing all the archimedean places, for every $k\in K^\times$ and for any
binary homogeneous form 
$F(X,Y)$
with the property that the polynomial $F(X,1) \in K[X]$ has at least three linear factors involving three distinct roots in $K$, 
the Thue-Mahler equation 
$$
F(X,Y)=kE
$$ 
has but a finite number of classes of solutions $(x,y,\varepsilon)\in\OS^2\times\OS^\times $.
 \hfill\break
\null\quad
$(ii)$ For any finite set $S$ of places of $K$ containing all the archimedean places, the Thue-Mahler equation 
$$
XY(X-Y)=E
$$ 
has but a finite number of classes of solutions $(x,y,\varepsilon)\in\OS^2\times\OS^\times $.
 \hfill\break
\null\quad
$(iii)$ For any finite set $S$ of places of $K$ containing all the archimedean places, the $S$--unit equation 
$$
E_1+E_2=1
$$ 
has but a finite number 
of solutions $(\varepsilon_1,\varepsilon_2)$ in $\OS^\times \times\OS^\times $. 
 \hfill\break
\null\quad
$(iv)$ For any finite set $S$ of places of $K$ containing all the archimedean places, every set of $S$--integral points of $\P^1(K)$ minus three
 points is finite. 
 \hfill\break 
 \hfill\break 
 $(2)$ 
Moreover, each of these assertions is a consequence of the following one: 
 \hfill\break
\null\quad
$(v)$ For any finite set $S$ of places of $K$ containing all the archimedean places, every set $A$ of $S$--integral points on an open variety 
$\bV$, obtained by removing from $\P^2(K)$ four hyperplanes, 
is contained in a finite union of projective hyperplanes of $\P^2(K)$. 
 \end{proposition} 
 
Before proceeding with the proof, many remarks are in order. 
These assertions are true: $(i)$ to $(iv)$ are theorems essentially going back to the work of K. Mahler 
 (\cite{MR0249355}; 
 \cite{MR88h:11002}, Chap.~7; 
 \cite{MR93a:11048}, Chap.~IX \S3; 
\cite{UZ1}, Chap.~I \S4 and \S5,   Chap.~III \S2; 
 \cite{MR2465098}, \S8.1; 
[ \cite{UZ2} Chap.~2).  
In $(iii)$ the finiteness statement for the number of solutions of the unit equation was singled out by C.L. Siegel, K. Mahler and S. Lang.  
The assertion $(iv)$ (resp{.} $(v)$) is the particular case $n=1$ (resp{.} $n=2$) of a theorem on integral points of $\P^n(K)$ minus $n+2$ hyperplanes, which in turn is a special case of a theorem due to P. Vojta concerning integral points on a variety minus a suitable divisor (see \S$\ref{S:Unit-Vojta}$). 
Moreover, the three missing points in $(iv)$ are classically denoted 
\begin{equation}\label{Equation:ZeroUnInfini}
{\bf 0}=(0:1), \;\; {\bf 1}=(1:1), \; \; {\boldsymbol \infty}=(1:0).
\end{equation}
It should now be clear that the spirit of the last proposition is to state that the truth of each of the first four assertions implies the truth of each of the three other ones,
and to state that the truth of the fifth assertion implies the truth of each of the first four assertions. 
We give elementary proofs of the equivalences of some assertions, while the proof of the truth of each of these assertions
relies on Schmidt’s Subspace Theorem.


Here again, like in \S$\ref{S:ApproximationEquations}$ and \S$\ref{S:texteJapon}$, explicit versions are known for each of the  statements $(i)$ to $(iv)$ in Proposition $\ref{Proposition:equivalence}$, and the proofs of the equivalences between these assertions enable one to deduce, from an explicit version of one of them, an explicit version for each of the three other statements.

 The remarkably powerful Subspace Theorem of W. Schmidt
 generated vast generalisations of these five assertions together with the statements of Proposition $\ref{Proposition:Thue}$. The methods of C.L. Siegel, F. Dyson, Th. Schneider, K.F. Roth and W.M. Schmidt are not effective. They allow us to give upper bounds for the number of solutions or of classes of solutions, but one had to wait till the major breakthrough of A. Baker  
  (\cite{BakerTNT}, \S~4.5; 
 \cite{LangECDA}, Chap.~VI;  
\cite{57:12383}, Chap.~3; 
 \cite{MR88h:11002}, Chap.~7; 
 \cite{SerreMordellWeil},  Chap.~8), %
 to obtain explicit bounds for the solutions themselves, which bounds we cannot avoid when we want to solve completely these equations.

 The $S$--unit equation $E_1+E_2=1$ in assertion $(iii)$ is in a non--homogeneous form. The associated homogeneous $S$--unit equation is $E_1+E_2=E_3$, a special case of the generalized Siegel unit equation which will be considered in \S$\ref{S:Unit-Vojta}$.

 \medskip 
\noindent
{\bf Definition.} 
Two solutions $(\varepsilon_0,\dots,\varepsilon_n)$ and $(\varepsilon'_0,\dots,\varepsilon'_n)$ in $(\OS^\times)^{n+1}$ of the equation $E_0+\cdots+E_n=0$ are said to be equivalent modulo $\OS^\times$ if the points of $\P^n(K)$ with projective coordinates $(\varepsilon_0:\cdots:\varepsilon_n)$ and $(\varepsilon'_0:\cdots:\varepsilon'_n)$ are the same. 

This last property means that there exists $\eta\in\OS^\times$ such that 
$$
\varepsilon'_j=\eta\varepsilon_j \quad \hbox{for } \quad 0\le j\le n.
$$

 \medskip
\noindent
{\it Proof of Proposition $\ref{Proposition:equivalence}$. } 
If the homogeneous form $F$ of degree $n \ge 3$ in assertion $(i)$ is such that $F(X,1)$ has at least three
linear factors involving three distinct roots $\alpha_1,\alpha_2,\alpha_3$ in $K$,
then there exists a homogeneous form $H(X, Y ) \in K[X, Y ]$ of degree $n - 3 \ge 0$ such that
\begin{equation}\label{Equation:H}
F(X,Y)=(X-\alpha_1Y)(X-\alpha_2Y)(X-\alpha_3Y)H(X,Y),
\end{equation}
 where the polynomial $H(X,1)$ needs not be 
  a monic polynomial and may have its
roots outside $K$ (though assuming $H(X,1)$ to be monic with roots in $K$ would not restrict the generality).
Moreover, we let $\denominator\in \Z$ be a positive integer such that $\denominator H \in \ZK[X,Y]$.

 We are going to prove the implications
 $$
 (i)\Longrightarrow (ii)\Longrightarrow (iii) \Longrightarrow (i) \quad\hbox{and}\quad 
 (iii)\Longleftrightarrow (iv) 
 \quad\hbox{and}\quad (v)\Longrightarrow (iii). 
 $$ 
This will complete the proof of Proposition $\ref{Proposition:equivalence}$.

\par
\begin{proof}[\quad$ (i)\Longrightarrow(ii) $] 
We make a change of variables 
 $$
 X'=X-Y, \quad Y'=X+Y
 $$
and we apply $(i)$
to the cubic form $F(X',Y')=X'(X'-Y')(X'+Y')$.
\end{proof}

\par
\begin{proof}[\quad$ (ii)\Longrightarrow(iii) $] 
Let $(\varepsilon_1,\varepsilon_2)\in\OS^\times$ satisfy $\varepsilon_1+\varepsilon_2=1$. Set $x=1$ and $y=\varepsilon_1$, so that 
$$
xy(x-y)=\varepsilon_1 \varepsilon_2.
$$
Each class modulo $\OS^\times$ of solutions $(x,y,\varepsilon)\in\OS^2\times\OS^\times$ of $XY(X-Y)=E$ contains a unique element with the first component $1$, namely $(1, x^{-1}y,x^{-3}\varepsilon)$. Since there is a finite number of classes of solutions, the set of $(1,\varepsilon_1,\varepsilon_1\varepsilon_2)$ with $\varepsilon_1+\varepsilon_2=1$ is finite, hence there is only a finite number of $\varepsilon_1\!$'s in $\OS^\times$ such that $1-\varepsilon_1\in\OS^\times$. 
\end{proof} 

\par
\begin{proof}[\quad$ (iii)\Longrightarrow(i) $]
Suppose that the assertion $(iii)$ is true and that $(x,y,\varepsilon)\in\OS^2\times\OS^\times$ is a solution of the equation $F(X,Y)=kE$. Write, as in $(\ref{Equation:H})$, 
$$
F(X,Y)=(X-\alpha_1Y)(X-\alpha_2Y)(X-\alpha_3Y)H(X,Y),
$$
where $\alpha_1,\alpha_2,\alpha_3$ are three roots of $F(X,1) $ which are distinct and in $K$.

 Define $\beta_i=x-\alpha_i y$ ($i=1,2,3$) so that $\beta_1\beta_2\beta_3H(x,y)=k\varepsilon$. Then we eliminate
 $x$ and $y$ from these three linear relations defining $\beta_1$, $\beta_2$ and $\beta_3$ to obtain the homogeneous unit equation (already considered by Siegel)
\begin{equation}\label{Equation:unites}
(\alpha_1-\alpha_2)\beta_3+(\alpha_2-\alpha_3)\beta_1+(\alpha_3-\alpha_1)\beta_2=0.
\end{equation}
 Define $S$ to be the set of places given by the assertion $(i)$, and apply $(iii)$ with the set $S'$ obtained by adding to $S$ the places of $K$ dividing numerators and denominators of the fractional principal ideals $(k)$, $(\denominator )$, $(\alpha_i-\alpha_j)$ ($1\le i<j\le 3$). Hence the three terms $(\alpha_i-\alpha_j)\beta_k$ of the left member of $(\ref{Equation:unites})$ are $S'$--units. We deduce from $(iii)$ that the quotients $\beta_i/\beta_j$ ($i,j=1,2,3$) belong to a fixed finite set, say, $\{\gamma_1, \dots , \gamma_t\} $ which is independent of the solution $(x,y,\varepsilon)$ considered. Suppose that $\beta_2 = \gamma \beta_1$ with $\gamma\in \{\gamma_1, \dots , \gamma_t\} $. Set $\eta=\beta_1$ (recall that $\beta_1$ is an $S'$--unit), 
 $$
x_0=\frac{\alpha_1 \gamma - \alpha_2}{\alpha_1 -\alpha_2} , \quad 
y_0= \frac{\gamma -1}{\alpha_1 -\alpha_2} 
\quad\hbox{
and
}
\quad 
\varepsilon_0=
k^{-1} F(x_0,y_0).
$$
Then from the values of $\beta_1$ and of $\beta_2$ ($ = \gamma \beta_1$), we obtain
$$
x=x_0\eta,\quad y=y_0\eta,\quad \varepsilon=\varepsilon_0\eta^n.
$$
We deduce that modulo $\OSprime^\times$ there is only a finite number of classes of solutions of $F(X,Y)=kE$.
This allows us to conclude that the assertion $(i)$ is true. 
\end{proof}

\par
\begin{proof}[\quad$ (iv)\Longrightarrow(iii) $]
Let $\calE$ be the set of $(\varepsilon_1,\varepsilon_2)\in\OS^\times\times\OS^\times$ for which $\varepsilon_1+\varepsilon_2=1$. 
Then the set
$$
\{
(\varepsilon_1:1) \, \mid \, \hbox{there exists $\varepsilon_2\in\OS^\times$ such that $(\varepsilon_1,\varepsilon_2)\in \calE$}\}
$$
is a set of $S$--integral points of $\P^1(K)\setminus\{
 {\bf 0}, {\bf 1},{\boldsymbol \infty}\}$, where 
$ {\bf 0}, {\bf 1},{\boldsymbol \infty}$ are defined in $(\ref{Equation:ZeroUnInfini})$, 
 hence it is finite by $(iv)$, and $(iii)$ follows. 
 \end{proof} 
 
 \par
\begin{proof}[\quad$ (iii)\Longrightarrow(iv) $] 
Let $A$ be a set of $S$-integral points $(x:y)$ 
on $ \P^1(K)$ minus three points chosen (without loss of generality) to be 
$ {\bf 0}, {\bf 1},{\boldsymbol \infty}$, as defined in $(\ref{Equation:ZeroUnInfini})$. Since $A$ is contained in $\P^1(K)\setminus\{(1:0)\}$, each element $P$ in $A$ has projective coordinates $(u:1)$ with $u\in K$. Since $P$ does not reduce modulo a finite place $v$ not in $S$ to any of the three points $(1:0)$, $(0:1)$, $(1:1)$, 
it follows that $u$ and $u':=1-u$ are $S$--units. From $u+u'=1$ we deduce from $(iii)$ that the set of such $u$'s is finite, hence $A$ is finite. 
\end{proof}

\par
\begin{proof}[\quad$ (v)\Longrightarrow(iii) $]
Consider the system of projective coordinates $(E:E_1:E_2)$ on $\P^2(K)$ and  the four hyperplanes $H_0$, $H_1$, $H_2$, $H_3$ defined respectively by the equations
$$
\hbox{	$E=0$, \;$E_1=0$,\; $E_2=0$,\; $E_1+E_2=0$.}
$$
Let $\calE$ be the subset of $(\OS^\times)^2$ which consists of the couples $(\varepsilon_1,\varepsilon_2)$ of $S$--units verifying $ \varepsilon_1+\varepsilon_2=1$. For any $\varepsilon$ in $\OS^\times$, the point of $P^2(K)$ with projective coordinates $(\varepsilon:\varepsilon_1:\varepsilon_2)$ is an $S$--integral point of $\P^2(K)\setminus(H_0\cup H_1\cup H_2\cup H_3)$. Indeed, such a point reduces modulo each place $v$ of $K$ not in $S$ to a point on the projective plane over the residue field which is not on any of the four corresponding hyperplanes. 
We deduce from $(v)$ the existence of a non--zero homogeneous polynomial $P(E,E_1,E_2) $ in $K[E,E_1,E_2]$ which is annihilated by each of the points in $\OS^\times\times \calE$.
Assuming (without loss of generality) that $\OS^\times$ is infinite, it follows that for each $(\varepsilon_1, \varepsilon_2)\in\calE$ the polynomial $P(E,\varepsilon_1, \varepsilon_2)\in K[E]$ is the zero polynomial, 
 whereupon the assertion $(iii)$ is true.
 \end{proof} 
 
 This concludes the proof of the fact that indeed the first four assertions of Proposition $\ref{Proposition:equivalence}$ are equivalent to one another and are consequences of the fifth one. 
 \hfill
 $\Box$
 
 \medskip

It may be a fruitful goal to devise further proofs of direct implications between the assertions of Proposition $\ref{Proposition:equivalence}$: taking shortcuts may be useful for further investigations,
and we hope that the proofs of these implications are interesting {\it per se}. In particular, there are at least two points of view for obtaining sharper statements, and for each of them there is a whole variety of methods, involving deep and powerful tools from Diophantine approximation. Firstly, by having an effective statement via an explicit upper bound for the number of solutions or of classes of solutions. Secondly, by giving an upper bound for the height of the solutions, which is the effective way of dealing with the theory. When it comes to establishing such explicit versions of those mentioned implications, using no detours may prove a winning strategy to obtain more precise bounds.
 This is why we now prove directly the next implication.

\par
\begin{proof}[\quad $(ii)\Longrightarrow (i) $]
Suppose that the assertion $(ii)$ is true. We want to prove $(i)$ for a homogeneous binary form of degree $n$ that we write as in $(\ref{Equation:H})$. Change the variables as follows: set
$$
X'=(\alpha_2-\alpha_3)(X-\alpha_1Y), \quad 
Y'=(\alpha_1-\alpha_3)(X-\alpha_2Y), 
$$ so that
$$
X'-Y'=(\alpha_2-\alpha_1)(X-\alpha_3Y).
$$
Given the set $S$ of $(i)$, we will use the set $S'$ of $(ii)$ which is the union of $S$ with the set of places of $K$ dividing numerators and denominators of the fractional principal ideals $(\denominator )$, $(\alpha_1)$, $(\alpha_2)$, $(\alpha_3)$, 
$(\alpha_2-\alpha_3)$, $(\alpha_1-\alpha_3)$, $(\alpha_2-\alpha_1)$ and $(k)$, and also the principal ideals generated by the coefficients of the form $\denominator H$. 

If $(x,y,\varepsilon)\in\OS^2\times\OS^\times$ satisfies $F(x,y)=k\varepsilon$ where $F$ is given by $(\ref{Equation:H})$, then the corresponding elements $x'$, $y'$ obtained by the change of variables are $S'$--integers with the property that the number $\varepsilon':=x'y'(x'-y')$ is an $S'$--unit. 
 The assertion $(ii)$ provides the finiteness of the set of classes modulo $\OSprime^\times$ of solutions $(x',y',\varepsilon')$ in $\OSprime^2\times \OSprime^\times$ of the equation $X'Y'(X'-Y')=E' $.
 Since the matrix attached to the above change of variables has determinant
$(\alpha_1 -\alpha_2)(\alpha_1-\alpha_3)(\alpha_2-\alpha_3)\neq 0$, we deduce that the assertion $(i)$ is true. 
 \end{proof} 
 
 From the equivalence between $(i)$ and $(ii)$, we deduce that these two properties are also equivalent to the special case of $(i)$ where one assumes $H=1$, (hence the form $F$ is a cubic form with $F(X,1)$ a monic polynomial), so that 
 $$
F(X,Y)=(X-\alpha_1Y)(X-\alpha_2Y)(X-\alpha_3Y)\in K[X,Y]
$$ 
and where one assumes also $k=1$. 
 
 \medskip
 We conclude this section with the remark that it would be very interesting to produce a proof of $(v)$ as a consequence of the previous assertions; (we already pointed out that all of these assertions, including $(v)$, are theorems). 
Indeed, the assertion $(v)$ has further far reaching consequences, besides assertions $(i)$ to $(iv)$. In particular it can be used to prove that any homogeneous diophantine unit equation 
$$
E_1+E_2+E_3+E_4=0
$$
has only finitely many solutions $(\varepsilon_1,\varepsilon_2,\varepsilon_3,\varepsilon_4)$ in $S$--units for which none of the three subsums
$$
\varepsilon_1+\varepsilon_2,\quad \varepsilon_1+\varepsilon_3,\quad \varepsilon_1+\varepsilon_4
$$
vanishes. So far, no effective proof of this result has been produced in general, while effective versions of assertions $(i)$ to $(iv)$ are known.

\section{Generalized Siegel unit equation and integral points} \label{S:Unit-Vojta}

In this section we prove the equivalence between two main Diophantine results, both consequences of Schmidt's Subspace Theorem 
  (\cite{MR0568710}, Chap.~6;	
 \cite{UZ1}, Chap.~II, \S1; 
 \cite{1115.11034}; 
 \cite{MR2465098}, \S7.5; 
 \cite{UZ2}, Chap.~2).  
Again the proof of the equivalence is elementary, while the proof of the truth of each of them lies much deeper. 

 
 \begin{proposition}\label{Proposition:EquivalenceUnitEqnHyperplans}
Let $K$ be a number field. 
The following two assertions are equivalent. 
\par
$(i)$ 
Let $n\ge 1$ be an integer and let $S$ a finite set of places of $K$ including the archimedean places. Then the equation 
$$
E_0+\cdots+E_n=0
$$ 
has only finitely many classes modulo $\OS^\times$ of solutions $(\varepsilon_0,\ldots,\varepsilon_n)\in(\OS^\times)^{n+1}$ for which no proper subsum $\sum_{i\in I} \varepsilon_i$ vanishes, with $I$ being a subset of $\{0,\ldots,n\}$, with at least two elements and at most $n$. 
\par
$(ii)$
Let $n\ge 1$ be an integer and let $S$ a finite set of places of $K$ including the archimedean places. Then for any set of $n+2$ distinct hyperplanes $H_0,\dots,H_{n+1}$ in $\P^n(K)$, the set of $S$--integral points of $\P^n(K)\setminus(H_0\cup\cdots\cup H_{n+1})$ is contained in a finite union of 
hyperplanes 
of $\P^n(K)$. 
 \end{proposition} 
 
 One may remark that the case $n=1$ of assertion $(i)$ in Proposition $\ref{Proposition:EquivalenceUnitEqnHyperplans}$ is nothing else than assertion $(iii)$ of Proposition $\ref{Proposition:equivalence}$, and that 
 the case $n=1$ (resp. $n=2$) of assertion $(ii)$ of Proposition $\ref{Proposition:EquivalenceUnitEqnHyperplans}$ is nothing else than assertion $(iv)$ (resp. $(v)$) of Proposition $\ref{Proposition:equivalence}$.

 Assertion $(i)$ of Proposition $\ref{Proposition:EquivalenceUnitEqnHyperplans}$ on the generalized unit equation (see 
  \cite{MR883451}, 
 Theorem 2.3.1;
 \cite{UZ1},  
Chap.~II, \S2 and \S3; 
 \cite{MR2465098}, 
Theorem 7.24) 
 has been proved independently by J.H.~Evertse on the one hand, by H.P.~Schlickewei and A.J.~van der Poorten (1982) on the other hand. A special (but significant) case had been obtained earlier by E.~Dubois and G.~Rhin
 (see \cite{UZ1}, 
 Chap.~II, \S2). 

There is a more general version of the assertion $(i)$ of Proposition $\ref{Proposition:EquivalenceUnitEqnHyperplans}$, which is known to be true, where the number field is replaced by any field $K$ of zero characteristic, and the group of $S$--units is replaced by any subgroup of $K^\times$ of finite rank. The first general result in this direction is due to M.~Laurent; it has been extended by Schmidt, Evertse, van der Poorten and Schlickewei (see 
\cite{1115.11034}, 
\S7.4), and recently refined by Amoroso and Viada (Theorem 6.2 of \cite{AV}).

In his thesis on integral points on a variety (1983), P.~Vojta started a fertile analogy between Diophantine approximation and Nevanlinna theory. In the case of holomorphic functions, the analog of assertion $(i)$ of Proposition $\ref{Proposition:EquivalenceUnitEqnHyperplans}$ is a result of E.~Borel in 1896 (see 
\cite{MR883451} 
Chap.~2, \S4) according to which, if $g_1,\ldots,g_n$ are entire functions satisfying $e^{g_1}+\cdots+e^{g_n}=1$, then some $g_i$ is constant. 
A connection between assertion $(i)$ of Proposition $\ref{Proposition:EquivalenceUnitEqnHyperplans}$ on $S$--units and integral points on the complement in a projective space of a divisor was found by P.~Vojta. In 1991, Min Ru and Pit Man Wong considered the case when the divisor is a union of $2n+1$ hyperplanes in general position and showed that the set of $S$--integral points is finite. Independently, K.~Gy\H{o}ry proved the same result in 1994, but formulated it in terms of decomposable form equations (see {\it e.g.} 
\cite{MR2438848}, 
 p.~261 for the dictionary between decomposable form equations and integral points on the complements of hypersurfaces). Further related results are due to 
Ta Thi Hoai An, Julie Tzu-Yueh Wang, Zhihua Chen, and more recently Aaron Levin (see \cite{MR2438848}
).

Assertion $(ii)$ of Proposition $\ref{Proposition:EquivalenceUnitEqnHyperplans}$ may be seen as a theorem on integral points which partially extends Siegel's Theorem to higher dimensional varieties 
\cite{MR883451}. 


It is proved in \cite{MR961695}, Section 4, that  $(i)$  of Proposition $\ref{Proposition:EquivalenceUnitEqnHyperplans}$,  and hence $(ii)$  as well, are equivalent to a
general finiteness theorem concerning decomposable form equations over $\OS$. (This equivalence is proved in \cite{MR961695} in the more general case when the ground ring is an arbitrary finitely generated domain over  $\Z$.

No effective version of the assertions $(i)$ and $(ii)$ is known. On the one hand, if one could prove an effective version of one of these two assertions, the proof we give for the equivalence between them would provide an effective version of the other. On the other hand, quantitative estimates are known, namely explicit upper bounds for the number of exceptions. The proof of the equivalence between $(i)$ and $(ii)$ shows also that an explicit upper bound for the number of exceptional classes in assertion $(i)$ yields an explicit upper bound for the number of exceptional hyperplanes in $(ii)$, and conversely. 

 In the proof of $(ii)\Longrightarrow (i)$, we shall use the following auxiliary result. Denote by $L_0$ the hyperplane of $\P^n(K)$ of equation $X_0+\cdots+X_n=0$ and, for $i=0,\dots,n$, by $H_i$ the hyperplane of equation $X_i=0$. 

\begin{lemma}\label{Lemma:52implique51}
Let $L$ be a projective line of $\P^n(K)$ contained in $L_0$.  Assume that $L$ contains a point of projective coordinates   $(u_0:\cdots:u_n)$ such that $u_0\cdots u_n\not=0$. Assume further that no sum $\sum_{i\in I} u_i$ vanishes, when $I$ is a subset of $\{0,\dots,n\}$ with at least one and at most $n$ elements.  Then among the $n+1$ subspaces 
$$
L\cap H_0,\dots,L\cap H_n,
$$
at least $3$ are distinct. 
\end{lemma}

From the assumption that 
$u_i\not =0$ for $0\le i\le n$,  it follows that for $0\le i\le n$, the line $L$ is not contained in $H_i$, and therefore $L\cap H_i$ is a point of $L$. 

For us, the useful consequence of Lemma $\ref{Lemma:52implique51}$ is the following one.

\begin{corollary}\label{Corollaire:52implique51}

Let $L$ be a projective linear subspace of $\P^n(K)$ contained in $L_0$. Let $s$ be the dimension of $L$.
Assume that $L$ contains a point of projective coordinates $(u_0 : \cdots: u_n)$ such that $u_0 \cdots u_n \neq  0$
and such that no subsum $\sum_{i\in I }u_i$ vanishes, with $I$ being a subset of $\{0,\ldots,n\}$, with at least two elements and at most $n$.  Then for any $s=0,\dots,n$, at least $s+2$ hyperplanes of $L$ among 
$$
L\cap H_0,\dots,L\cap H_n
$$
are distinct. 
\end{corollary}

\begin{proof}

This corollary 
is trivial when $s=0$, that is when $L$ is a point, since $H_0 \cap H_1\cap \dots \cap H_n=\emptyset$. 
It  follows from 
Lemma $\ref{Lemma:52implique51}$ when $s=1$. 

Assume now $2 \le  s \le  n -1$.  
Suppose there are at most $s+1$ distinct hyperplanes $L\cap H_0,\dots ,L\cap H_n$. Given one of these hyperplanes, there exists a point $\bv$ which does not belong to this hyperplane but belongs to all the other ones. 
Let $L'$ be a line  through $\bv$ and the given $(u_0 :\cdots  : u_n) \in L$. Then $L'$  will intersect  $H_0, \dots, H_n$  in  at most $2$ points, contradicting Lemma $\ref{Lemma:52implique51}$.
\end{proof}

\begin{proof}[Proof of Lemma $\ref{Lemma:52implique51}$] 
The goal is to check that among the points 
$$
L\cap H_0,\; L\cap H_1,\; \dots, L\cap H_n,
$$
at least $3$ are distinct. It is obvious that there are at least two points,  because $H_0 \cap H_1\cap \dots \cap H_n=\emptyset$.
Let ${\bv}=(v_0:\cdots:v_n)$ be a point on $L$ distinct from ${\bu}=(u_0:\cdots:u_n)$, so that 
$$
L=\{x\bu+y\bv=(xu_0+yv_0:\cdots:xu_n+yv_n)\, \mid \, x,\; y \in K\}.
$$
If there are only two points among the intersections $L\cap H_i$, $(0\le i\le n)$, after reordering the indices, we may suppose $L\cap H_0 = L\cap H_1 = \cdots= L\cap H_t$ and $L\cap H_{t+1} = \cdots= L\cap H_n$ with $1\le t<n$. One deduces 
\begin{align}\notag
L\cap H_0 = L\cap H_1 = \cdots= L\cap H_t&=\{(0:\cdots:0:u_{t+1}:\cdots:u_n)\}
\\
\notag
&=\{(0:\cdots:0:v_{t+1}:\cdots:v_n)\}
\end{align}
and 
\begin{align}\notag
L\cap H_{t+1} = \cdots= L\cap H_n&=\{(u_0:u_1:\cdots:u_t:0:\cdots:0)\}
\\
\notag
&=\{(v_0:v_1:\cdots:v_t:0:\cdots:0)\}.
\end{align}
The condition $L\subset L_0$ then implies $u_0+u_1+\cdots+u_t=0$ and $u_{t+1}+\cdots+u_n=0$, a contradiction with the condition on the non--vanishing of proper subsums. 
\end{proof}

\par
\begin{proof}[Proof of Proposition $\ref{Proposition:EquivalenceUnitEqnHyperplans}$] 
$(i)\Longrightarrow (ii)$. 
Let $(X_0:\cdots:X_n)$ be a system of projective coordinates on $\P^n(K)$ and let $L_0,\ldots,L_{n+1}$ be homogeneous linear forms in $X_0,\dots,X_n$ such that, for $i=0,\ldots,n+1$, the hyperplane $H_i$ is given by the equation $L_i=0$. Let $r+1$ be the rank of the system of linear forms $L_0,\ldots,L_{n+1}$. Reorder the forms so that $L_0,\ldots,L_r$ are linearly independent, and such that $L_{r+1}$ can be written as $a_0L_0+\cdots+a_mL_m$ with $m\le r$ and $a_0,\dots,a_m$ non--zero elements in $K$. Let $Y_j=a_jL_j$ for $0\le j\le m$. 
Complete $Y_0,\dots,Y_m$ in order to get a new system of projective coordinates $Y_0,\ldots,Y_n$ on $\P^n(K)$.
We apply assertion $(i)$ of Proposition $\ref{Proposition:EquivalenceUnitEqnHyperplans}$ to the projective subspace $\P^m(K)$ of $\P^n(K)$ given by the equation $Y_{m+1}=\cdots=Y_n=0$ and to the $m+2$ hyperplanes 
$$
Y_0=0,\; \dots,\; Y_m=0,\; Y_0+\cdots+Y_m=0.
$$ 
The map 
$$
(y_0:\cdots:y_n)\longmapsto (y_0:\cdots:y_m)\in\P^m(K)
$$ 
is well defined on $\P^n(K)\setminus H_0$ (recall that $H_0$ is the hyperplane of equation $Y_0=0$), hence also on the set of $S$--integral points of $\P^n(K)\setminus(H_0\cup\cdots\cup H_{n+1})$.
An $S$--integral point of $\P^n(K)\setminus(H_0\cup\cdots\cup H_{n+1})$ has projective coordinates $(y_0:\cdots:y_n)$ such that $y_0,\ldots,y_m$ and $y_0+\cdots+y_m$ are $S$--units, hence by assertion $(i)$ of Proposition $\ref{Proposition:EquivalenceUnitEqnHyperplans}$, for all but a finite number of the projective points $(y_0:\cdots:y_m)$, the tuple $(y_0,\ldots,y_m)$ has a vanishing proper subsum. Therefore the set of $S$--integral points of $\P^n(K)\setminus(H_0\cup\cdots\cup H_{n+1})$ is contained in the union of finitely many linear subspaces. 
\par
$(ii)\Longrightarrow (i)$. 
Let us introduce the subset 
 $\Etilde$ of $\P^n(K)$ which consists of the points having projective coordinates $(\varepsilon_0:\cdots:\varepsilon_n)$ with $\varepsilon_i\in \OS^\times$, $\varepsilon_0+\cdots+\varepsilon_n=0$, and no subsum in the left hand side with at least two and at most $n$ terms being $0$. The goal is to prove that this set $\Etilde$ is finite. 
 By induction, we prove the following consequence of $(ii)$. 
 
\begin{quote}
$(*)$ 
{\it For $k=0,1,\dots,n-1$, there exist a finite set $J_k$ and linear projective spaces $L_{k,j}$ ($j\in J_k$), of dimensions $n-k-1$, such that $\Etilde$ is contained in the union of $L_{k,j}$ for $j\in J_k$.}
\end{quote}
This assertion $(*)$ is true for $k=0$ with $J_0=\{0\}$ and $L_{0,0}=L_0$. 

We wish to prove that for $ k=0,\dots,  n-2$, the assertion $(*)$ for $k$ implies the same $(*)$ for $k+1$. Fix $j\in J_k$. We deduce from  Corollary 
$\ref{Corollaire:52implique51}$  with $s=n-k-1$ that at least $s+2$ hyperplanes of $L_{k,j}$ among 
$$
L_{k,j}\cap H_0,\dots,L_{k,j}\cap H_n
$$
are distinct. Next we deduce from   assertion $(ii)$ of Proposition $\ref{Proposition:EquivalenceUnitEqnHyperplans}$ for the space $L_{k,j}$ (with $n$ replaced by $s$) that $\Etilde\cap L_{k,j}$ is contained in a finite union of hyperplanes of $L_{k,j}$, the dimension of which is $s-1=n-(k+1)-1$. 
Denote by $\{L_{k+1,\ell}   \; | \;   \ell\in J_{k+1} \}$   the set of all these hyperplanes of the subspaces $L_{k,j}$ where $j$ ranges over  $J_k$. 
The truth of the assertion $(*)$ for $k+1$ follows. Finally the truth of $(*)$ with $k=n-1$ implies that $\Etilde$ is contained in the finite union of the subspaces $L_{n-1,j}$, ($j\in J_{n-1}$), where each $L_{n-1,j}$ has dimension $0$, hence is a point, and the finiteness of $\Etilde$ follows. 

\end{proof}

 \begin{rem}{\rm 
Assertion $(ii)$ of Proposition $\ref{Proposition:EquivalenceUnitEqnHyperplans}$ is different from Corollary 2.4.3 of Vojta in 
\cite{MR883451}: 
our hyperplanes $H_i$ are replaced by hypersurfaces, and Vojta's conclusion is that the set of $S$--integral points on the complement is 
degenerate (contained in a hypersurface). 
Vojta deduces his result from a more general result (Theorem 2.4.1 of 
\cite{MR883451}), 
according to which the set of $D$--integral points on a variety $V$ is degenerate, when $D$ is a divisor which is a sum of at least $\dim V + \varrho + r+1$ distinct prime divisors $D_i$. Here, $r$ is the rank of the group of rational points on the variety $\Pic^0(V)$ and $\varrho$ is the Picard number of $V$. For $\P^n(K)$ we have $r=0$ and $\varrho=1$. 
The proof of that result boils down  to the unit equation considered in assertion (i) 
of Proposition $\ref{Proposition:EquivalenceUnitEqnHyperplans}$, so again    Schmidt’s Subspace Theorem comes into play.

There are generalizations and improvements to Theorem 2.4.1 of \cite{MR883451}  given by Vojta in \cite{V2}, Corollary 0.3  and by Noguchi and Winkelmann in \cite{NW}.

}
\end{rem}

\section{Potpourri}\label{S:potpourri}

We conclude by including a few remarks which originated from comments we received on a preliminary version of the paper. 

 
 \begin{proposition}
 The assertion $(v)$ of Proposition $\ref{Proposition:equivalence}$ implies the assertion $(i)$.
 \end{proposition}
 
 \begin{proof}[Proof {\rm (after P. Corvaja)}]
Let $(X:Y:E)$ be a system of projective coordinates on $\P^2(K)$. Denote by $\calE$ the set of solutions $(x,y,\varepsilon)$ in $\OS^2\times\OS^\times$ of the equation $F(X,Y)=kE$, where $F$ is given by $(\ref{Equation:H})$. Consider the four hyperplanes $H_0$, $H_1$, $H_2$, $H_3$ of $\P^2(K)$ of equations respectively given by 
 $$
 E=0,\quad X-\alpha_1Y=0,\quad X-\alpha_2Y=0,\quad X-\alpha_3Y=0.
 $$
 Let $S'$ be the set obtained by adding to $S$ the places of $K$ dividing numerators and denominators of the fractional principal ideals $(k)$, $(\alpha_1)$, $(\alpha_2)$ and $(\alpha_3)$. Then for each $(x,y,\varepsilon)\in\calE$, the point in $\P^2(K)$ with projective coordinates $(x:y:\varepsilon)$ is an $S'$--integral point of $\P^2(K)\setminus (H_0\cup H_1\cup H_2\cup H_3)$. From Proposition $\ref{Proposition:equivalence}$ $(v)$ we deduce that there exists a non--zero homogeneous polynomial $P\in K[X,Y,E]$ such that $P(x,y,\varepsilon)=0$ for all $(x,y,\varepsilon)\in\calE$. For any $(x,y,\varepsilon)\in\calE$ and any $\eta\in\OS^\times$, we have $(\eta x,\eta y,\eta^m\varepsilon)\in\calE$, hence $P(\eta x,\eta y,\eta^m \varepsilon)=0$. Since $P$ is homogeneous, assuming (without loss of generality) that $\OS^\times$ is infinite, we deduce $P(x,y,E)=0$. Therefore the set of points $(x : y) \in \P^1(K)$
such that there exists $\varepsilon\in \OS^\times$ with $(x,y,\varepsilon) \in \calE$ is finite.
 \end{proof}
 
\begin{proposition}
The assertion $(v)$ of Proposition $\ref{Proposition:equivalence}$ implies the assertion $(iii)$.
 \end{proposition}

\begin{proof}[Proof 1 {\rm (after U. Zannier)}]
 Assume that the set $\calE$ of $(\varepsilon_1,\varepsilon_2)\in(\OS^\times)^2$ such that $\varepsilon_1+\varepsilon_2=1$ is infinite. Let $(\varepsilon_1,\varepsilon_2)$ and $(\eta_1,\eta_2)$ be two elements in $\calE$. From $\varepsilon_1+\varepsilon_2=1$ and $\eta_1+\eta_2=1$ one deduces 
 $$
 1-\varepsilon_2-\varepsilon_1\eta_2=\varepsilon_1\eta_1.
 $$
 Hence the point with projective coordinates $( 1:-\varepsilon_2:-\varepsilon_1\eta_2)$ is an $S$--integral point of $\P^2(K)\setminus(H_0\cup H_1\cup H_2\cup H_3)$, where $H_0$, $H_1$, $H_2$, $H_3$ are the hyperplanes of equations respectively given by  
 $$
 X_0=0,\quad X_1=0,\quad X_2=0,\quad X_0+X_1+X_2=0.
 $$ 
 Assume now the truth of assertion $(v)$ of Proposition $\ref{Proposition:equivalence}$: there exists a non--zero polynomial $P\in K[X,Y]$ such that $P(\varepsilon_2,\varepsilon_1\eta_2)=0$ for all $((\varepsilon_1,\varepsilon_2),(\eta_1,\eta_2))\in\calE^2$. Since $\calE$ is infinite, there are infinitely many $\eta_2$, hence the polynomial $P(\varepsilon_2,\varepsilon_1 T)$ is the zero polynomial, which implies $P(\varepsilon_2,Y)=0$, and since there are infinitely many $\varepsilon_2$, we obtain the contradiction $P=0$. 
 \end{proof}

\begin{proof}[Proof 2 {\rm (Geometrical proof, after U. Zannier)}]
The map 
 $$
\bigl( (X_0:X_1)\; , (Y_0:Y_1) \bigr)\longmapsto  (X_0Y_0 \, : \, X_1Y_0 \, : \,X_0Y_1 \, : \, X_1Y_1) 
$$
is a quadratic embedding of the square $\P^1(K)\times \P^1(K)$ in $\P^3(K)$ (with a system of projective coordinates $(T_0: T_1: T_2: T_3) $) as the quadratic surface $\calS$ of equation $T_0T_3=T_1T_2$. The image of $(\P^1(K)\setminus\{{\bf 0}, {\bf 1}, {\boldsymbol \infty}\})\times (\P^1(K)\setminus\{{\bf 0}, {\bf 1}, {\boldsymbol \infty}\})$ is $\calS$ minus the intersection of $\calS$ with the union of the six lines $L_0$, $L_1$, $L_2$ and $M_0$, $M_1$, $M_2$ of equations respectively given by 
$$
T_0=T_2=0, \quad T_1=T_3=0,\quad T_1-T_0=T_3-T_2=0
$$
and
$$
T_0=T_1=0, \quad T_2=T_3=0,\quad T_2-T_0=T_3-T_1=0.
$$
The point of intersection of $L_0$ and $M_0$ is $(0:0:0:1)$. 
The map $$
(t_0:t_1:t_2:1)\longmapsto (t_0:t_1:t_2)
$$ 
is a projection from 
$\P^3(K)\setminus \{(0:0:0:1)\}$ onto $\P^2(K)$. The projections in $\P^2(K)$ of the lines $L_1$, $L_2$, $M_1$, $M_2$ are four different lines and we apply the assertion $(v)$ to the complement of these four lines in $\P^2(K)$. Let $\calE$ be the set of $\varepsilon$ in $\OS^\times$ such that $1-\varepsilon$ is in $\OS^\times$. For $(\varepsilon,\eta)\in\calE^2$, the point $(\varepsilon\eta:\varepsilon:\eta)$ is an $S$--integral point of $\P^2(K)$ minus these four lines, hence there is a homogeneous polynomial which vanishes on all the points $(\varepsilon\eta,\varepsilon,\eta)$ with $(\varepsilon,\eta)\in\calE^2$. It follows that $\calE$ is finite. 
 \end{proof}
 
 \begin{proposition}
 The assertion $(v)$ of Proposition $\ref{Proposition:equivalence}$ implies the assertion $(iv)$.
 \end{proposition}
 
 
 \begin{proof}[Proof 1 {\rm (after P. Corvaja)}]
Let $E$ be a set of $S$--integral points of $\P^1(K)\setminus\{{\bf 0}, {\bf 1}, {\boldsymbol \infty}\}$. Take some systems of projective coordinates $(X_0:X_1)$ on $\P^1(K)$ and $(X_0:X_1:X_2)$ on $\P^2(K)$. Remove from $\P^2(K)$ the $4$ hyperplanes $H_0$, $H_1$, $H_2$, $H_3$ of equations $X_0=0$, $X_1=0$, $X_2=0$ and $X_1=X_0$. For any element in $E$ of projective coordinates $(1:\varepsilon)$ and for any $\eta\in\OS^\times$ with the property that $1-\varepsilon \in \OS^\times$, the point $(1:\varepsilon:\eta)$ is an $S$--integral point of $\P^2(K)\setminus\{H_0\cup H_1\cup H_2\cup H_3\}$. Hence the set of these points is contained in an algebraic hypersurface, and we deduce that $E$ is finite. 
 \end{proof}

 \begin{proof}[Proof 2 {\rm (after P. Corvaja)}]
In $\P^2(K)$ consider the $5$ hyperplanes $H_0$, $H_1$, $H_2$, $H_3$, $H_4$ of equations $X_0=0$, $X_1=0$, $X_2=0$, $X_1=X_0$, $X_2=X_0$. Let $E$ be the set of $\varepsilon\in K^\times$ such that $(1:\varepsilon)$ is an $S$--integral point of $\P^1(K)\setminus\{{\bf 0}, {\bf 1}, {\boldsymbol \infty}\}$. Then for any pair $(\varepsilon_1,\varepsilon_2)$ of elements in $E\times E$, the point of projective coordinates $(1: 
\varepsilon_1:\varepsilon_2)$ is an $S$--integral point of $\P^2(K)\setminus\{H_0\cup H_1\cup H_2\cup H_3\cup H_4\}$, hence $E\times E$ is contained in an algebraic hypersurface, and it follows that $E$ is finite. 
 \end{proof}

\begin{proposition}\label{proposition:Zannier}
Let $n$ and $t$ be integers with $1\le t < n$.
 The truth of assertion $(i)$ of Proposition $\ref{Proposition:EquivalenceUnitEqnHyperplans}$ for $n$ implies the truth of the result for $n-t$.
\end{proposition}

\begin{proof}[Proof {\rm (after U. Zannier)}]
Denote by $\calE$ the set of $(\varepsilon_0,\dots,\varepsilon_{n-t})$ in $(\OS^\times)^{n-t+1}$ satisfying 
\begin{equation}\label{Equation:nmoinst}
\varepsilon_0+\cdots+\varepsilon_{n-t}=0
\end{equation}
with the non--vanishing of any proper subsum of the left hand side. Let $\gamma$ be an element in $\OS\setminus\OS^\times$,  with the property that  $r/\gamma \not\in \OS$ for $r=1, \dots, t$. Let $S'$ be the set obtained by adding to $S$ the places of $K$ dividing $\gamma(\gamma-t)$. Write the left hand side of $(\ref{Equation:nmoinst})$ as 
$$
\gamma\varepsilon_0+\cdots+ \gamma\varepsilon_{n-t-1}+(\gamma-t)\varepsilon_{n-t}+
\underbrace{\varepsilon_{n-t}+ \cdots +\varepsilon_{n-t}}_{t \; {\rm times}}
$$
and consider it  as a sum of 
$n+1$ elements which are $S'$--units of $K$. 

That no proper subsum is $0$ follows from the following four  remarks:

(i) Since a proper subsum of the sum in $(\ref{Equation:nmoinst})$ cannot be $0$, for any non--empty subset $\{i_1,\ldots,i_m\}$ of $\{0,\dots,n-t-1\}$, we have $\gamma(\varepsilon_{i_1}+\cdots +\varepsilon_{i_m} ) \not= 0$. 

(ii) For $s \geq 0 $, we have $(\gamma -s) \varepsilon_{n-t} \not = 0$.

(iii) Since $s/\gamma $ is not an $S$-integer, for any non--empty subset $\{i_1,\ldots,i_m\}$ of $\{0,\dots,n-t-1\}$ and any $0\leq s \leq t$, we have 
$\gamma( \varepsilon_{i_1} + \cdots + \varepsilon_{i_m}+\varepsilon_{n-t} ) \not = s\varepsilon_{n-t}$. 

(iv) For the same reason, for any non--empty subset $\{i_1,\ldots,i_m\}$ of $\{0,\dots,n-t-1\} $ and any $0\leq s \leq t$, we have 
$\gamma( \varepsilon_{i_1} + \cdots + \varepsilon_{i_m}+\varepsilon_{n-t} ) + s\varepsilon_{n-t}\not = 0$.

\par
Assuming that assertion $(i)$ of Proposition $\ref{Proposition:EquivalenceUnitEqnHyperplans}$ is true for $n$, it follows that $\calE$ is a union of finitely many equivalent classes modulo $\OSprime^\times$, hence modulo $\OS^\times$.
\end{proof}

\begin{proposition}
Let $n$ and $t$ be integers satisfying  $1\le t < n$.
 The truth of assertion $(ii)$ of Proposition $\ref{Proposition:EquivalenceUnitEqnHyperplans}$ for $n$ implies the truth of the result for $n-t$.
\end{proposition}

\begin{proof}[Proof {\rm (after G.~Rémond)}] 
Using the same argument as in the proof $(i)\Longrightarrow (ii)$ of Proposition $\ref{Proposition:EquivalenceUnitEqnHyperplans}$,
we deduce that there is a system of  projective coordinates $(X_0:\cdots :X_n)$ on $\P^n(K)$ and there is an integer $r$ in the range $1\le r \le t$ such that $(X_0:\cdots :X_{n-t})$ is a system of  projective coordinates on $\P^{n-t}(K)$ and $n-r+2$ of the given hyperplanes in $\P^{n-t}(K) $ are defined by the equations 
$X_0 =0$, $X_1 =0$, $\dots$, $X_{n-r} =0$ and $X_0 +\cdots +X_{n-r} =0$. Let ${\cal E}$ be the set of $S$–integral points on the complements in $\P^{n-r}(K) $ of these hyperplanes. 
Consider the hyperplanes $X_0 =0$, $X_1 =0$, $\dots$, $X_n =0$ and $X_0 +\cdots +X_{n-r} =0$
of $\P^n(K)$. Assuming that  assertion $(ii)$ of Proposition $\ref{Proposition:EquivalenceUnitEqnHyperplans}$  holds for $n$, we deduce that there exists a homogeneous polynomial
$Q$ in $n+1$ variables which vanishes at $(\varepsilon_0, \dots, \varepsilon_{n-r}, \eta_1, \dots , \eta_r)$ for all $(\varepsilon_0, \dots, \varepsilon_{n-r})$ in ${\cal E}$
and all $(\eta_1, \dots , \eta_r)$ in $(\OS^\times)^r$. If $\OS^\times$ is infinite, then the polynomial
$Q(\varepsilon_0,\dots,\varepsilon_{n-r},X_1, \dots , X_r) $ does not depend on $X_1, \dots , X_r$ and we deduce that assertion $(ii)$ of Proposition $\ref{Proposition:EquivalenceUnitEqnHyperplans}$ holds for $n - t$.
\end{proof}


The next proposition follows from Proposition $\ref{proposition:Zannier}$: we give another proof of it. 

\begin{proposition} 
The truth of assertion $(i)$ of Proposition $\ref{Proposition:EquivalenceUnitEqnHyperplans}$ for a fixed $n\ge 3$ implies the truth of the result for $n=2$.
\end{proposition}

 \begin{proof}[Proof {\rm (after U. Zannier)}]
 %
Let $n\ge 3$. Set $m=n-1$. Let $\varepsilon\in\OS^\times$ satisfy $1-\varepsilon\in\OS^\times$. Write
$$
(1- \varepsilon)^m -m \varepsilon + \dots +(-1)^j \binom{m}{j} \varepsilon^j +\dots +(-1)^m \varepsilon^m=1.
$$
Let $S'$ denote the set obtained by adding to $S$ the places of $K$ dividing the binomial coefficients $\displaystyle \binom{m}{j}$ for $1\le j\le m-1$. The left hand side is a sum of $m+1$ terms which are $S'$--units. The set of $S$--units $\varepsilon$ for which there is a vanishing proper subsum is finite, namely it is the set of roots of finitely many polynomials of the form 
$$
u_0(1- E)^m -u_1 m E + \cdots +(-1)^j u_j \binom{m}{j} E^j 
+\cdots +(-1)^m u_m E^m,
$$
where $ u_t\in \{0,1\}$ for $t=0, \dots , m$. 
From the assumption that assertion $(i)$ of Proposition $\ref{Proposition:EquivalenceUnitEqnHyperplans}$ is true for $n$, we deduce that the set of these $S$--units $\varepsilon$ is finite. 
\end{proof}
 

\section*{Acknowledgements}

We received a number of comments on a preliminary version of this paper. We are grateful to 
Francesco Amoroso,
Pascal Autissier,
Pietro Corvaja, 
Jan-Hendrik Evertse, 
Kalman Gy\H{o}ry,
Gaël Rémond,
Paul Vojta,
Isao Wakabayashi 
and
Umberto Zannier
for the valuable suggestions they kindly gave us. 
Last and not least, we thank the anonymous referee for his generous and detailed  remarks and his   pertinent  suggestions.


\bigskip
 
 \noindent
{\sc Claude Levesque
}\\
{D\'{e}partement de math\'{e}matiques et de statistique,
}\\
{Universit\'{e} Laval,
}\\
{Qu\'{e}bec (Qu\'{e}bec),
}\\
{CANADA G1V 0A6
}\\
{\href{Claude.Levesque@mat.ulaval.ca}{\tt Claude.Levesque@mat.ulaval.ca}
 }

\bigskip

 \noindent
{\sc Michel Waldschmidt
}\\
{Institut de Math\'{e}matiques de Jussieu, 
}\\
{Universit\'{e} Pierre et Marie Curie (Paris 6),
}\\
{4 Place Jussieu, 
}\\
{F -- 75252 PARIS Cedex 05, FRANCE
}\\
{\href{miw@math.jussieu.fr}{\tt miw@math.jussieu.fr}
 }\\
{\href{http://www.math.jussieu.fr/~miw/}{\tt http://www.math.jussieu.fr/$\sim$miw/}
} 
 \vfill\vfill

 \end{document}